 \documentclass[final,1p]{elsarticle}
\usepackage{amssymb}
\biboptions{sort&compress}
\usepackage{amsthm}
\usepackage{bm}
\usepackage{amsmath}
\usepackage{enumerate}
\usepackage{enumitem}
 \usepackage{lineno}
\usepackage[caption=false]{subfig}
\newtheorem{theorem}{Theorem}[section]

\newtheorem{lemma}[theorem]{Lemma}
\newtheorem{example}[theorem]{Example}

\begin{document}

\begin{frontmatter}



\title{A semi-discrete numerical scheme for nonlocally regularized KdV-type  equations}


\author[label1]{H. A. Erbay\corref{cor1}}
    \ead{husnuata.erbay@ozyegin.edu.tr}
    \address[label1]{Department of Natural and Mathematical Sciences, Faculty of Engineering, Ozyegin University,  Cekmekoy 34794, Istanbul, Turkey}
    \cortext[cor1]{Corresponding author}

\author[label1]{S. Erbay}
    \ead{saadet.erbay@ozyegin.edu.tr}

\author[label2]{A. Erkip}
    \ead{albert@sabanciuniv.edu}
    \address[label2]{Faculty of Engineering and Natural Sciences, Sabanci University, Tuzla 34956, Istanbul,  Turkey}

\begin{abstract}
A general class of KdV-type  wave equations regularized with a convolution-type nonlocality in space is considered.    The class  differs from the class of the nonlinear nonlocal unidirectional wave equations previously studied  by the addition of a  linear convolution term involving third-order derivative. To solve the Cauchy problem we propose a  semi-discrete numerical method based on a uniform spatial discretization, that is  an extension of a previously published work of the present authors.  We  prove uniform convergence of the numerical method  as the mesh size goes to zero. We also prove that the localization error  resulting from localization to a finite domain is significantly less than  a given threshold if the finite domain is large enough. To illustrate the theoretical results,  some numerical experiments  are carried out for the Rosenau-KdV equation, the Rosenau-BBM-KdV equation  and a convolution-type integro-differential equation. The experiments conducted for three particular choices of the kernel function confirm the error estimates that we provide.
\end{abstract}


\end{frontmatter}


\setcounter{equation}{0}
\section{Introduction}\label{sec:sec1}

In this paper, we will propose a semi-discrete numerical method for nonlocally regularized Korteweg-de Vries-type equations of the form
\begin{equation}
     u_{t} +\alpha \ast \big((f(u))_{x}+\kappa u_{xxx}\big)=0, \label{eq:cont}
\end{equation}%
where  $\kappa$ is  a positive constant and $\alpha$  is a sufficiently smooth kernel of the convolution operator
 \begin{displaymath}
    (\alpha \ast v)(x)= \int_{\mathbb{R}} \alpha(x-y)v(y)\mbox{d}y.
\end{displaymath}
Obviously,   (\ref{eq:cont}) reduces to the  Korteweg-de Vries (KdV) equation \cite{Korteweg1895} when $\alpha$ is the Dirac measure and $f(u)=u^{2}/2$. The basic idea of the numerical method is to consider space discretization of (\ref{eq:cont}) on a uniform grid and  to discretize the convolution integral. We  prove uniform convergence of the numerical method and  show that the accuracy  depends upon the smoothness properties of the kernel function. We provide numerical experiments  that illustrate  our error estimates and also make comparisons with the exact solutions available only in very special cases of $\alpha$.

Although various aspects of nonlinear dispersive wave models have been studied extensively in the literature, there are fewer studies on nonlocal  models of wave propagation  appearing in many  applications.   The class (\ref{eq:cont}) considered in this article has not been previously addressed,  but some members of this class are well-known equations in the literature. For instance, if the kernel function $\alpha$ is chosen as the  Green's function
\begin{equation}
    \alpha(x)= {1\over {2\sqrt 2}}e^{-{\vert x\vert\over \sqrt 2}}
                        \Big( \cos\big({{\vert x\vert}\over {\sqrt 2}}\big) + \sin \big({{\vert x\vert}\over {\sqrt 2}}\big) \Big ) \label{eq:ros-ker}
\end{equation}
of the differential operator $1+D^{4}_{x}$  (where $D_{x}$ represents the partial derivative with respect to $x$),   (\ref{eq:cont}) reduces to a generalized form of the  Rosenau-Korteweg de Vries (Rosenau-KdV) equation  \cite{Wang2019} (see also \cite{Rosenau1988})
\begin{equation}
    u_{t} +u_{x}+u_{xxxxt}+\kappa u_{xxx}+(g(u))_{x}=0 \label{eq:rosenau}
\end{equation}
for $f(u)=u+g(u)$. On the other hand, if $\alpha$ is chosen as the  Green's function
\begin{equation}
    \alpha(x)= {1\over {2\sqrt 3}}e^{-{\sqrt 3\over  2}\vert x\vert}\Big( \cos\big({{\vert x\vert}\over  2}\big) +\sqrt 3 \sin \big({{\vert x\vert}\over  2}\big) \Big ) \label{eq:ros-bbm-ker}
\end{equation}
of the differential operator $1-D^{2}_{x}+D^{4}_{x}$, then we get  from  (\ref{eq:cont}) the Rosenau-BBM (Benjamin-Bona-Mahony)-KdV equation \cite{ Wongsaijai2014}
 \begin{equation}
    u_{t} +u_{x}-u_{xxt}+u_{xxxxt}+\kappa u_{xxx}+(g(u))_{x}=0.  \label{eq:rosenau-bbm}
\end{equation}
Finally, if we let $\alpha(x)={1\over 2}e^{-|x|}$ which is the  Green's function of the differential operator $1-D^{2}_{x}$, (\ref{eq:cont}) reduces to  a generalized form of the BBM-KdV equation,
\begin{equation}
    u_{t}+u_{x}-u_{xtt}+\kappa u_{xxx}+(g(u))_{x}=0.  \label{eq:bbm}
\end{equation}
Keeping these familiar examples in mind, we propose a numerical method for (\ref{eq:cont}) by imposing some assumptions on $\alpha$, that are compatible with (\ref{eq:ros-ker}) and (\ref{eq:ros-bbm-ker}). It still remains, however, to extend the present analysis to the kernels that are compatible with the exponential kernel.

We note that (\ref{eq:cont})  does not reduce to a partial differential equation unless $\alpha$ is  the Green's function of a differential operator. The most important property of the present  semi-discrete scheme is that it can be successfully used to solve (\ref{eq:cont})  with  an arbitrary kernel function  regardless of whether it is a Green's function. The reason for this improvement is that the scheme  is based on truncated discrete convolution sums rather than finite-difference approximations.

Taking inspiration from the numerical scheme  used in \cite{Bona1981} to solve the   BBM equation  \cite{Benjamin1972}, in previous two works the present authors have proposed the semi-discrete numerical methods based on both  a uniform space discretization and the discrete convolution operator  for the nonlocal nonlinear bidirectional wave equation  \cite{Erbay2018} and for the  nonlocal nonlinear unidirectional wave equation  \cite{Erbay2021}. The present work is an extension of   \cite{Erbay2021} (in which $\kappa=0$) to   (\ref{eq:cont}) involving the extra "KdV term" $u_{xxx}$.   So we shall not give some proofs in full detail, but instead refer the reader to \cite{Erbay2018, Erbay2021} for more details.    As in  \cite{Erbay2021} our strategy  in obtaining  the  discrete problem is to transfer the spatial derivative in (\ref{eq:cont}) to the kernel function.  As it was already observed in \cite{Bona1981}, an advantage of this direct approach is that  a further time discretization will not involve any stability issues regarding spatial mesh size.  Note that for a sufficiently smooth kernel, the derivatives can be transferred to $\alpha $ and the equation (\ref{eq:cont}) becomes
\begin{equation*}
        u_{t}+\alpha ^{\prime }\ast f(u)+\kappa \alpha ^{\prime \prime \prime }\ast u=0.
\end{equation*}%
Since  the spatial discrete derivatives of $u$ do not appear in the resulting discrete problem, even though this extension involves the extra "KdV term" $u_{xxx}$, the discretizations and proofs are more straightforward due to the extra smoothness assumption on the kernel $\alpha $. Our approach will also apply to nonlocal equations of the form
\begin{equation*}
     u_{t} +\alpha_{1} \ast (f(u))_{x}+\kappa \alpha_{2}\ast u_{xxx}=0,
\end{equation*}%
as well as generalizations involving higher-order derivatives under appropriate smoothness conditions  on the kernels.

The paper is structured as follows. In Section \ref{sec:sec2}  we discuss some of the basics of setting up, such as,   local well-posedness of the continuous Cauchy problem and discretization in space.  In Section \ref{sec:sec3},  the convergence of the discretization error with respect to mesh size for the semi-discrete problem    is proved. In Section \ref{sec:sec4}  we discuss the localization error for  the truncated problem and give a decay estimate. Section \ref{sec:sec5} is devoted to  numerical experiments.

The notation used in the present paper is as follows.  $\Vert u\Vert_{L^p}$ is the $L^p$ ($1\leq p \leq \infty$) norm of $u$ on $\mathbb{R}$, $W^{k,p}(\mathbb{R})=\{ u\in L^p(\mathbb{R}): D^ju \in L^p(\mathbb{R}),~~ j \leq k \} $ is the $L^{p}$-based Sobolev space with the norm $\Vert u\Vert_{W^{k,p}}=\sum_{j\leq k} \|D^j u\|_{L^p},~1\leq p \leq \infty$ and  $H^{s}$ is the usual $L^{2}$-based  Sobolev space of index $s$ on $\mathbb{R}$.  $C$ denotes a generic positive constant. For a real number $s$, the symbol $[s]$ denotes the largest integer less than or equal to $s$.

\setcounter{equation}{0}
\section{Preliminaries}\label{sec:sec2}

In this section we  first  give some preliminary results about  local well-posedness of the Cauchy problem for  (\ref{eq:cont1}) and  error estimates of discretizations of integrals on an infinite uniform grid.

\subsection{The Continuous  Cauchy Problem}
We  consider the Cauchy problem
\begin{align}
  &   u_{t}+\alpha ^{\prime }\ast f(u)+\kappa\alpha ^{\prime \prime \prime }\ast  u=0, \text{ \ \ \ }  x\in \mathbb{R}\text{, \ \ }t>0, \label{eq:cont1}  \\
  &  u(x,0)=\varphi(x), \text{ \ \ \ }  x\in \mathbb{R}.  \label{eq:initial}
  \end{align}%
We suppose that $f$ is sufficiently smooth with $f(0)=0$ and that the kernel $\alpha$ is to be chosen to satisfy  the following two constraints:
\begin{enumerate}[label={ C\arabic*.}, align=left]
 {    \item $\alpha \in W^{2,1}(\mathbb{R})$,

    \item $\alpha ^{\prime\prime\prime}=\mu $ is a finite Borel measure on $\mathbb{R}$.}
\end{enumerate}

At this point, it should be noted that Condition C2 also includes the more regular case $\alpha \in W^{3,1}(\mathbb{R})$ (i.e. $\alpha^{\prime\prime\prime}\in  L^{1}(\mathbb{R})$) in which case $d\mu=\alpha ^{\prime\prime\prime}dx$. The following theorem states  the local well-posedness of solutions to (\ref{eq:cont1})-(\ref{eq:initial}).
\begin{theorem}\label{theo:theo2.1}
    Suppose that $\alpha$ satisfies Conditions  C1 and C2. Let $s>1/2$, $f\in C^{[s] +1}(\mathbb{R})$ with $f(0)=0$. For a given $\varphi \in H^{s}(\mathbb{R})$, there is some $T>0$ so that the initial-value problem (\ref{eq:cont1})-(\ref{eq:initial}) is locally well-posed with solution $u\in C^{1}\left([0,T],H^{s}(\mathbb{R})\right)$, satisfying the estimate
    \begin{equation}
        {\left\Vert u(t)\right\Vert }_{H^{s}}\leq \left\Vert \varphi \right\Vert_{H^{s}}e^{Ct},~~~~0\leq t\leq T,
        \label{growth23}
    \end{equation}
    where $C$ depends on  $M=\sup_{0\leq t\leq T}\left\Vert u(t)\right\Vert_{L^{\infty}}$ and $\alpha$.
\end{theorem}
The proof of Theorem \ref{theo:theo2.1} follows from Picard's theorem  for Banach space-valued ODEs. The  estimate (\ref{growth23})  is derived from the $L^{\infty}$ control of the nonlinear term (see  \cite{Duruk2010} for details). This estimate, in particular, implies that blow-up in finite time is also determined by the $L^{\infty}$ norm of the solution.

\subsection{Discretization}

As in \cite{Erbay2018} and   \cite{Erbay2021}, we  consider doubly infinite sequences $\mathbf{w}=(w_{i})_{i=-\infty}^{i=\infty}=(w_{i})$ of real numbers  $w_{i}$ with $i\in\mathbb{Z}$  (where $\mathbb{Z}$ denotes the set of integers). Let $h>0$ denote the spatial mesh size to be used in the
discretization. The  $l_{h}^{p}(\mathbb{Z})$ space  for  $1\leq p<\infty $ is
\begin{displaymath}
   l_{h}^{p}\left(\mathbb{Z}\right)=\left\{ (w_{i}): w_{i}\in \mathbb{R}, ~~
        \Vert \mathbf{w}\Vert_{l_{h}^{p}}^{p}=\sum_{i=-\infty }^{\infty}h|w_{i}|^{p}\right\}.
\end{displaymath}
The $l^{\infty}(\mathbb{Z})$ space with the sup-norm $\displaystyle \Vert \mathbf{w}\Vert_{l^{\infty}}=\sup_{i \in\mathbb{Z}} \left \vert w_{i} \right \vert$ is a Banach space.  The discrete convolution operation denoted by the symbol $*$ transforms two sequences  $\mathbf{w}$ and $\mathbf{v}$ into a new sequence:
\begin{equation}
    (\mathbf{w}\ast \mathbf{v})_{i}=\sum_{j}hw_{i-j}v_{j} \label{eq:disc-con}
\end{equation}%
(from now on, the symbol $\displaystyle \sum_{j}$ will be used to denote summation over all $j\in\mathbb{Z}$). Young's inequality for convolution integrals state that $\Vert \mathbf{w}\ast \mathbf{v}\Vert_{l_{h}^{p}}\leq \Vert \mathbf{w}\Vert _{l_{h}^{1}}\Vert \mathbf{v}\Vert _{l_{h}^{p}}$ for
$\mathbf{w}\in l_{h}^{1}, \mathbf{v}\in l_{h}^{p}$, $1\leq p <\infty $ and $\Vert \mathbf{w}\ast \mathbf{v}\Vert _{l^{\infty} }\leq \Vert \mathbf{w}\Vert_{l_{h}^{1}}\Vert \mathbf{v}\Vert_{l^{\infty}}$ for $\mathbf{w}\in l_{h}^{1},\mathbf{v}\in l^{\infty }$.

Let $w$ be a function of one variable $x$ with domain $\mathbb{R}$. We begin with a uniform partition of the  real line and define   the grid points $x_{i}=ih$, $i\in \mathbb{Z}$   with the mesh size $h$. Let the restriction operator $\mathbf{R}$ be   $\mathbf{R}w=(w(x_{i}))$.  From now on we will use the abbreviations $\mathbf{w}$ and $\mathbf{w}^{\prime }$ for $\mathbf{R}w$ and  $\mathbf{R}w^{\prime }$,  respectively.

The next lemma provides an  error bound for the discrete (trapezoidal) approximations of the integral over $\mathbb{R}$ depending on the smoothness of the integrand.
\begin{lemma}[\cite{Erbay2018}]\label{lem:lem2.2}
        Let $w\in W^{1,1}(\mathbb{R})$ and  $w^{\prime\prime}=\nu$ be a finite measure  on $\mathbb{R}$. Then
        \begin{equation}
            \left\vert \int_\mathbb{R} w(x)dx-\sum_{i}h w(x_{i})\right \vert \leq h^{2}\vert \nu\vert (\mathbb{R}). \label{eq:intB}
        \end{equation}
   Moreover,  $\mathbf{w}=\mathbf{R}w\in l_{h}^{1}$ and $\Vert \mathbf{w}\Vert _{l_{h}^{1}}\leq \Vert w\Vert _{l^{1}}+h^{2}\vert\nu \vert (\mathbb{R})$.
\end{lemma}
Let  $D^{2}_{h}$ be the second-order difference operator  defined as
\begin{displaymath}
    \left(D^{2}_{h}\mathbf{u}\right)_{i}={1\over h^{2}}\left(u_{i+1}-2u_{i}+u_{i-1}\right).
\end{displaymath}
An error bound for  the approximation to the second derivative by  $D^{2}_{h}$   is reported in the following lemma.
\begin{lemma}\label{lem:lem2.3}
    Let $u\in W^{4,\infty}(\mathbb{R})$, $\mathbf{u}=\mathbf{R}u$ and $\mathbf{u}^{\prime \prime }=\mathbf{R}u^{\prime \prime }$. Then
    \begin{displaymath}
        \Vert D^{2}_{h}\mathbf{u}-\mathbf{u}^{\prime \prime }\Vert _{l^{\infty}}\leq \frac{h^{2}}{12}\Vert u^{(4)}\Vert _{L^{\infty}}.
    \end{displaymath}
\end{lemma}

\setcounter{equation}{0}
\section{ The Semi-Discrete Problem  and Discretization Error}\label{sec:sec3}

In this section we introduce the semi-discrete formulation of the Cauchy problem and prove convergence of solutions of the semi-discrete problem to solutions of the continuous Cauchy problem.

\subsection{The Semi-Discrete Problem}

We now formulate  the semi-discrete problem associated with (\ref{eq:cont1})-(\ref{eq:initial})  on a mesh with a fixed spatial mesh size $h>0$.  In that respect, it will be convenient to discretize the term $\alpha^{\prime\prime\prime}\ast u$  in  (\ref{eq:cont1})  by $D^{2}_{h}\bm{\alpha}^{\prime}_{h}\ast \mathbf{v}$ with the notation $\bm{\alpha}^{\prime}_{h}=\mathbf{R}\alpha^{\prime}$. The reason behind this is
\begin{equation*}
                \alpha^{\prime\prime\prime}\ast u=\alpha^{\prime}\ast u^{\prime\prime}
                        \approx \bm{\alpha}^{\prime}_{h}\ast D^{2}_{h} \mathbf{v}=D^{2}_{h}\bm{\alpha}^{\prime}_{h}\ast \mathbf{v}.
\end{equation*}
 Thus, the discretized form of the nonlocal wave equation (\ref{eq:cont1}) on the uniform infinite grid becomes
\begin{equation}
    \frac{d \mathbf{v}}{dt} =-\bm{\alpha}^{\prime}_{h}\ast f(\mathbf{v})-\kappa D^{2}_{h}\bm{\alpha}^{\prime}_{h}\ast \mathbf{v}   \label{eq:disc}
\end{equation}%
with the notation  $f(\mathbf{v})=(f(v_{i}))$. We estimate $D^{2}_{h}\bm{\alpha}^{\prime}_{h}$ as follows.
\begin{lemma}\label{lem:lem3.1}
    $D^{2}_{h}\bm{\alpha}^{\prime}_{h}\in l_{h}^{1}$ and $\Vert D^{2}_{h}\bm{\alpha}^{\prime}_{h}\Vert _{l_{h}^{1}}\leq 2|\mu| (\mathbb{R}).$
\end{lemma}
The proof of this lemma follows closely the proof of Lemma 3.4 of  \cite{Erbay2018}, so we skip it.

We note that  (\ref{eq:disc}) is an $l^{\infty}$-valued ODE system. The following theorem establishes  local well-posedness of solutions to the Cauchy problem.
\begin{theorem}\label{theo:theo3.2}
        Let $f$  be a locally Lipschitz function with $f(0)=0$. Then the initial-value problem for (\ref{eq:disc}) is locally well-posed for initial data $\mathbf{v}(0)$ in $l^{\infty}$. Moreover there exists some maximal time $T_{h}>0$ so that the problem has unique solution $\mathbf{v}\in C^{1}([0,T_{h}),l^{\infty})$. The maximal time $T_{h}$, if finite, is determined by the blow-up condition
         \begin{equation}
                \limsup_{t\rightarrow T_{h}^{-}}\Vert \mathbf{v}(t)\Vert_{l^{\infty} }=\infty .  \label{eq:blow2a}
        \end{equation}
\end{theorem}

\subsection{ An Estimate for the Discretization Error}

Let  $u\in C^{1}\left([0,T],H^{s}(\mathbb{R})\right)$ with sufficiently large $s$   be the unique solution of the continuous problem  (\ref{eq:cont1})-(\ref{eq:initial}). We denote the discretizations of the initial data $\varphi$   by  $\bm{\varphi}_{h}=\mathbf{R}\varphi$. Suppose that $\mathbf{u}_{h}\in C^{1}\left([0,T_{h}),l^{\infty}\right)$ is the unique solution of  the discrete problem based on (\ref{eq:disc}) and  the initial data $\bm{\varphi}_{h}$. In this subsection, our goal is to estimate the discretization error defined as $\mathbf{R}u(t)-\mathbf{u}_{h}$.
In the following theorem we prove that the discretization error is of $\mathcal{O}(h^2)$ for the case where  the kernel function $\alpha$ satisfies Conditions C1 and C2. The proof  follows  similar lines as the corresponding one in \cite{Erbay2018}.
\begin{theorem}\label{theo:theo3.3}
    Suppose that $\alpha$ satisfies Conditions C1 and C2.     Let $s>9/2$, $ f\in C^{[ s] +1}(\mathbb{R})$ with $f(0)=0$.  Let $u\in C^{1}\left([0,T], H^{s}(\mathbb{R})\right)$ be the solution of the initial-value problem (\ref{eq:cont1})-(\ref{eq:initial}) with $~\varphi \in H^{s}(\mathbb{R})$. Similarly, let  $\mathbf{u}_{h}\in C^{1}\left([0,T_{h}),l^{\infty}\right)$ be the solution of  (\ref{eq:disc}) with initial data $\bm{\varphi}_{h}$.   Let $\mathbf{u}(t)=\mathbf{R}u(t) =(u(x_{i},t))$. Then there is some $h_{0}$ so that for $h\leq h_{0}$, the maximal existence time $T_{h}$ of $\mathbf{u}_{h}$ is at least $T$ and
    \begin{equation}
        \Vert \mathbf{u}(t)-\mathbf{u}_{h}(t)\Vert_{l^{\infty}}=\mathcal{O}(h^{2})    \label{eq:fourone}
    \end{equation}
    for all $t\in \lbrack 0,T \rbrack$.
\end{theorem}
\begin{proof}
    Suppose that $\displaystyle M=\max_{0\leq t\leq T}\Vert u(t)\Vert _{L^{\infty}}$. Clearly  $\Vert \bm{\varphi }_{h}\Vert _{l^{\infty }}\leq \Vert \varphi \Vert_{L^{\infty }}\leq M$. By continuity there is some maximal time $t_{h}\leq T$ such that $\Vert \mathbf{u}_{h}(t)\Vert _{l^{\infty}}\leq 2M$ for all $t\in \lbrack 0,t_{h}]$.  By the maximality condition we must have either $t_{h}=T$ or $\Vert \mathbf{u}_{h}(t_{h})\Vert _{l^{\infty}}=2M$.  Evaluating (\ref{eq:cont1}) at  the grid points $x_{i}$ yields a set of differential equations
    \begin{displaymath}
            u_{t}(x_{i},t)+\big(\alpha ^{\prime }\ast f(u)\big)(x_{i},t)+\kappa \big((\alpha ^{\prime})^{\prime \prime }\ast u\big)(x_{i},t)=0.
    \end{displaymath}%
     Recalling that $\mathbf{u}(t)=\mathbf{R}u(t)$, this can also be expressed as
     \begin{equation*}
            \mathbf{u}^{\prime }(t)=-\mathbf{R}\Big( (\alpha ^{\prime }\ast f(u))(t)+\kappa ((\alpha ^{\prime})^{\prime \prime }\ast u)(t)\Big).
    \end{equation*}
     A residual term $\mathbf{F}_{h}$ arises from the discretization  of (\ref{eq:cont1}):
     \begin{equation}
                {\frac{d\mathbf{u}}{dt}}=-\bm{\alpha}^{\prime}_{h}\ast f(\mathbf{u})-\kappa D^{2}_{h}\bm{\alpha}^{\prime}_{h}\ast \mathbf{u}+\mathbf{F}_{h},
                \label{eq:u-F}
    \end{equation}
     where  $\mathbf{F}_{h}=\mathbf{F}^{1}_{h}+\mathbf{F}^{2}_{h}$ with
     \begin{displaymath}
               \mathbf{F}^{1}_{h}= \bm{\alpha}^{\prime}_{h}\ast f(\mathbf{u})-\mathbf{R}\big(\alpha ^{\prime }\ast f(u)\big),~~~~~
                \mathbf{F}^{2}_{h}= \kappa D^{2}_{h}\bm{\alpha}^{\prime}_{h}\ast \mathbf{u}-\mathbf{R}\big(\kappa (\alpha ^{\prime})^{\prime \prime }\ast u\big).
     \end{displaymath}
      Suppressing the $t$ variable,  the $i$-th entry of  $\mathbf{F}_{h}$  satisfies
       \begin{equation}
                (F_{h})_{i} =(F_{h}^{1})_{i}+(F_{h}^{2})_{i}   \label{eq:decomp}
        \end{equation}
      with
     \begin{eqnarray}
                (F_{h}^{1})_{i} &=& \sum_{j}h\alpha ^{\prime }(x_{i}-x_{j})f(u(x_{j}))-\int_\mathbb{R} \alpha^{\prime}(x_{i}-y)f(u(y))dy, \nonumber \\
                (F_{h}^{2})_{i} &=& \kappa \sum_{j}hD^{2}_{h}\alpha^{\prime}(x_{i}-x_{j})u(x_{j})
                                        -\kappa \int_\mathbb{R} (\alpha^{\prime})^{\prime \prime }(x_{i}-y)u(y)dy \nonumber.
    \end{eqnarray}
      We first estimate the  term $(F_{h}^{1})_{i}$. By  Lemma \ref{lem:lem2.2} we have
    \begin{displaymath}
       \left\vert (F_{h}^{1})_{i}\right\vert \leq h^{2}\left\Vert r^{\prime\prime}\right\Vert _{L^{1}},
    \end{displaymath}%
    where  $r(y)=\alpha^{\prime}(x_{i}-y)f(u(y))$. Then
    \begin{displaymath}
                r^{\prime\prime}(y)=\frac{d^2~}{dy^{2}}(\alpha^{\prime}(x_{i}-y))f(u(y))
                                        +2\frac{d~}{dy}(\alpha^{\prime}(x_{i}-y))\frac{d~}{dy}f(u(y))
                                                        +\alpha^{\prime}(x_{i}-y)\frac{d^2~}{dy^{2}}f(u(y)).
     \end{displaymath}%
      We have $\Vert f(u)\Vert_{H^{s}}\leq C(M)\Vert u \Vert_{H^{s}}$. For $s>9/2$ we observe that $\Vert u^{(k)}\Vert_{L^{\infty}}\leq C  \Vert u \Vert_{H^{s}}$ for $k\leq 4$. Then  $(f(u))^{\prime}$ and $(f(u))^{\prime\prime}$ are bounded. Since $\alpha^{\prime\prime\prime}=\mu$ is a finite measure, then $r^{\prime\prime}=\widetilde{\mu}$ will be a measure with
     \begin{displaymath}
                        \vert \widetilde{\mu} \vert (\mathbb{R})\leq C \Big(  \vert\mu \vert(\mathbb{R}) +2\Vert \alpha \Vert_{W^{2,1}}  \Big)\Vert u \Vert_{H^{s}}
      \end{displaymath}%
      so that
       \begin{displaymath}
                        \vert (F^{1}_{h})_{i} \vert \leq h^{2}\vert \widetilde{\mu} \vert (\mathbb{R}).
      \end{displaymath}%
      We now estimate the second term  $(F_{h}^{2})_{i}$.  Note that  $\mathbf{F}^{2}_{h}$ can be rewritten in the form
       \begin{displaymath}
             \mathbf{F}^{2}_{h}= \kappa \bm{\alpha}^{\prime}_{h}\ast D^{2}_{h} \mathbf{u}-\kappa \mathbf{R}\big(\alpha ^{\prime}\ast u^{\prime\prime}\big)
                                                =\kappa \bm{\alpha}^{\prime}_{h}\ast \big(D^{2}_{h} \mathbf{u}- \mathbf{R}u^{\prime\prime}\big).
     \end{displaymath}
     So, we have
      \begin{eqnarray*}
                    \vert (F_{h}^{2})_{i}\vert
                     &=&\Big\vert \Big(\kappa\bm{\alpha}^{\prime}_{h}\ast \big( D^{2}_{h}\mathbf{u}- \mathbf{R}u^{\prime\prime}\big)\Big)_{i}\Big\vert \\
                     &\leq &  \kappa   \Vert \bm{\alpha}^{\prime}_{h} \Vert_{l_{h}^{1}}~ \Vert D^{2}_{h}\mathbf{u}- \mathbf{R}u^{\prime\prime} \Vert_{l^{\infty} }
                          \leq   \frac{h^{2}}{12}\kappa \Vert \bm{\alpha}^{\prime}_{h} \Vert_{l_{h}^{1}} \Vert  u^{(4)} \Vert_{L^{\infty}} \leq Ch^{2} \Vert u\Vert_{H^{s}}
        \end{eqnarray*}
        where Lemma \ref{lem:lem2.3} and the Sobolev embedding theorem are used. If we combine the estimates for  $|(F_{h}^{1})_{i}|$ and $|(F_{h}^{2})_{i}|$, we get
        \begin{equation*}
                    \Vert \mathbf{F}_{h}\Vert_{l^{\infty }}\leq Ch^{2}\Vert u\Vert_{H^{s}},
        \end{equation*}
        where $C=C(\alpha ,\tilde{M})$ depends on the bounds on $\alpha $ and $\tilde{M}=\max_{0\leq t\leq T}\Vert u(t)\Vert _{H^{s}}$.    We now let $\mathbf{e}(t)=\mathbf{u}(t)-\mathbf{u}_{h}(t)$  be the error term. Then, from (\ref{eq:disc}) and (\ref{eq:u-F}) we have
        \begin{displaymath}
                {\frac{d\mathbf{e}}{dt}}=-\bm{\alpha }_{h}^{\prime }\ast (f(\mathbf{u})-f(\mathbf{u}_{h}))
                                                                -\kappa D^{2}_{h}\bm{\alpha }_{h}^{\prime}\ast (\mathbf{u}-\mathbf{u}_{h})+\mathbf{F}_{h}, ~~~~\mathbf{e}(0)=\mathbf{0}
        \end{displaymath}%
        or equivalently
         \begin{displaymath}
             \mathbf{e}(t)
                    = \int_{0}^{t}\Big( -\bm{\alpha }_{h}^{\prime }\ast (f(\mathbf{u})-f(\mathbf{u}_{h}))
                                -\kappa D^{2}_{h}\bm{\alpha }_{h}^{\prime}\ast (\mathbf{u}-\mathbf{u}_{h})+\mathbf{F}_{h}\Big) d\tau.
         \end{displaymath}
         Since $f$ is locally Lipschitz we have $\Vert f(\mathbf{u})-f(\mathbf{u}_{h})\Vert_{l^{\infty}} \leq C \Vert \mathbf{u}-\mathbf{u}_{h} \Vert_{l^{\infty}} $. So
           \begin{equation}
            \Vert \mathbf{e}(t)  \Vert_{l^{\infty}}
                    \leq Ch^{2}T+\big(C\Vert \bm{\alpha }_{h}^{\prime }\Vert_{l_{h}^{1}}+\kappa \Vert D^{2}_{h}\bm{\alpha }_{h}^{\prime} \Vert_{l_{h}^{1}}\big)  \int_{0}^{t}       \Vert \mathbf{e}(\tau)  \Vert_{l^{\infty}}  d\tau.  \label{eq:eqe}
         \end{equation}
        Then, by Lemma \ref{lem:lem3.1} and Gronwall's inequality, we get
        \begin{displaymath}
        \Vert \mathbf{e}(t)\Vert _{l^{\infty }}            \leq C h^{2} T             e^{C T}.
    \end{displaymath}
   This implies that, if $h$ is sufficiently small, we have   $\Vert \mathbf{e}(t_{h})\Vert _{l^{\infty }}<M$. Consequently we have $\Vert \mathbf{u}_{h}(t_{h})\Vert _{l^{\infty }}<2M$ showing that $t_{h}=T_{h}=T$. The above estimate yields  (\ref{eq:fourone}) or more explicitly $\Vert \mathbf{u}(t)-\mathbf{u}_{h}(t)\Vert_{l^{\infty}}\leq C(\alpha, u, f, T)h^{2}$. Obviously $C$ depends on  $\Vert \alpha\Vert_{W^{2,1}}$, $|\mu| (\mathbb{R})$, the solution $u$, the nonlinear term $f(u)$ and the existence time $T$.
\end{proof}

\setcounter{equation}{0}
\section{The Truncated Problem and a Decay Estimate}\label{sec:sec4}

In this section we introduce a finite dimensional approximation of the  semi-discrete problem and give  a suitable decay estimate on the tails of the solutions. We shall not review the detailed proof of the error estimate and the decay estimate because the basic ideas and proofs are given  fully  in \cite{Erbay2018, Erbay2021}.

\subsection{The Truncated Problem}
From now on we will assume that the infinite convolution sum defined by  (\ref{eq:disc-con}) was truncated at a finite $N$. Also we  will truncate the infinite system (\ref{eq:disc}) to the system of $2N+1$ equations to
obtain the finite-dimensional system
\begin{equation}
        \frac{dv_{i}^{N}}{dt}
                =-\sum_{j=-N}^{N}h\alpha ^{\prime}(x_{i}-x_{j})f(v_{j}^{N})
                    -\kappa \sum_{j=-N}^{N}hD^{2}_{h}\alpha ^{\prime}(x_{i}-x_{j})v_{j}^{N},\text{\ \ \ \ \ \ }-N\leq i\leq N  \label{eq:trunca}
\end{equation}%
where $v_{i}^{N}$ are the components of a vector valued function $\mathbf{v}^{N}(t)$ with finite dimension $2N+1$. We rewrite (\ref{eq:trunca}) as
\begin{displaymath}
    \frac{d\mathbf{v}^{N}}{dt}=A^{N}f(\mathbf{v}^{N})+B^{N}\mathbf{v}^{N},
\end{displaymath}
where $A^{N}$ and $B^{N}$ are the $(2N+1)\times (2N+1)$ matrices with the entries $a^{N}_{ij}=-h\alpha^{\prime} (x_{i}-x_{j})$   and $b^{N}_{ij}=-h\kappa D^{2}_{h}\alpha^{\prime} (x_{i}-x_{j})$, respectively. On the infinite interval, the only boundary condition is boundedness at infinity. It should be noted that our formulation on the finite range does not include any boundary terms.  By Lemma \ref{lem:lem3.1} we have
\begin{displaymath}
        \Vert A^{N}\mathbf{w}\Vert_{l^{\infty}}\leq \Vert \bm{\alpha }_{h}^{\prime }\Vert_{l^{1}_{h}}\Vert \mathbf{w}\Vert_{l^{\infty}}, ~~\text{and}~~
        \Vert B^{N}\mathbf{w}\Vert_{l^{\infty}}\leq  2|\mu| (\mathbb{R})\Vert \mathbf{w}\Vert_{l^{\infty}},
\end{displaymath}
respectively, where we use the norm $\displaystyle \Vert \mathbf{w}\Vert_{l^{\infty}}=\max_{-N \leq i \leq N} \left \vert w_{i} \right \vert$ for vectors in $\mathbb{R}^{2N+1}$. Since we assume that  $f$ is a locally Lipschitz and smooth function,  the initial-value problem defined for  (\ref{eq:trunca}) (as an ODE system) has a solution on $[0, T^{N})$. Also the blow-up condition
\begin{equation}
        \limsup_{t\rightarrow (T^{N})^{-}}\Vert \mathbf{v}^{N}(t)\Vert _{l^{\infty}}=\infty  \label{eq:blow3}
\end{equation}%
of the truncated problem is compatible with (\ref{eq:blow2a}) in the infinite discrete problem.

We will now estimate the localization error resulting from considering (\ref{eq:trunca}) instead of (\ref{eq:disc}). As the proofs are very similar to the ones in \cite{Erbay2018, Erbay2021}, we will only state the results. Consider  the projection  of the solution $\mathbf{v}$ of the semi-discrete problem associated with (\ref{eq:disc}) onto $\mathbb{R}^{2N+1} $. Let $\mathcal{T}^{N}:l^{\infty }\rightarrow \mathbb{R}^{2N+1}$ be the  truncation operator defined by $\mathcal{T}^{N}\mathbf{v}=(v_{-N},v_{-N+1},\ldots ,v_{0},\ldots ,v_{N-1},v_{N})$.  The following theorem estimates
 the localization error defined as $\mathcal{T}^{N}\mathbf{v}-\mathbf{v}^{N}$.
\begin{theorem}\label{theo:theo4.1}
        Let $\mathbf{v}\in C^{1}\left([0,T],l^{\infty}\right)$  be the solution of (\ref{eq:disc}) with initial value $\mathbf{v}(0)$ and let
         \begin{displaymath}
                \delta=\sup \big\{ \left\vert v_{i}(t)\right\vert : t\in \left[ 0,T\right] , \left\vert i\right\vert >N\big\} ~~\mbox{and}~~
                \epsilon (\delta )=\max_{\left\vert z\right\vert \leq \delta } \left\vert f(z)\right\vert .
        \end{displaymath}
        Then for sufficiently small $\epsilon(\delta)$, the solution $\mathbf{v}^{N}$ \ of (\ref{eq:trunca}) with initial  value $\mathbf{v}^{N}(0)={\mathcal T}^{N}\mathbf{v}(0)$ exists for  times $t\in \lbrack 0,T]$ and
        \begin{displaymath}
                \left\Vert \mathcal{T}^{N}\mathbf{v}(t)-\mathbf{v}^{N}(t)\right\Vert _{l^{\infty }}\leq C\epsilon (\delta ),\text{ \ }t\in \lbrack 0,T].
        \end{displaymath}%
\end{theorem}
 The next theorem shows that, for sufficiently large $N$, $\mathbf{v}^{N}$ approximates  the solution $u$ of the continuous problem (\ref{eq:cont1})-(\ref{eq:initial}).
\begin{theorem}\label{theo:theo4.2}
      Suppose that $\alpha$ satisfies Conditions C1 and C2.   Let $s>9/2$, $f\in C^{[ s] +1}(\mathbb{R})$ with $f(0)=0$.    Let $u\in C^{1}\left([0,T],H^{s}(\mathbb{R})\right)$  be the solution  of the initial-value problem (\ref{eq:cont1})-(\ref{eq:initial}) with $\varphi \in H^{s}(\mathbb{R})$.  Then for sufficiently small $h$ and  $\epsilon >0$, there is an $N$ so that  the solution $\mathbf{u}_{h}^{N}$ \ of (\ref{eq:trunca}) with initial  values $\mathbf{u}_{h}^{N}(0)={\mathcal T}^{N}\bm{\varphi}_{h}$, exists for  times $t\in \lbrack 0,T]$ and
        \begin{equation}
                \Big\vert u(ih,t) -\left(\mathbf{u}_{h}^{N}\right)_{i}(t)\Big\vert
                        = {\mathcal O}\left(h^{2}+\epsilon\right),          \text{ \ }  t\in \lbrack 0,T] \label{eq:fivetwo}
        \end{equation}%
        for all $-N\leq i\leq N$.
\end{theorem}
Since the proofs are very similar to those  in  \cite{Erbay2018, Erbay2021}, we skip the proofs of the above  theorems.

\subsection{A Decay Estimate}

We now ask whether there are values of $N$ for which the localization error $\epsilon$ is kept at a certain level. In Proposition 5.3 of \cite{Erbay2018} it has  been proven that such an $N$ exists, under the condition that the solution goes to zero at infinity. We note that, for a given level of $\epsilon>0$, the equation
\begin{displaymath}
        \epsilon=\max\big\{ |f(u(x,t))|: |x|\geq Nh, ~~t\in [0, T]\big\}
\end{displaymath}
provides  an implicit description of the solution set for $N$. A more explicit relation between $N$ and $\epsilon$ can be given if there are some decay estimates for the solution to the initial-value problem. The following lemma  provides a general decay estimate for the solutions corresponding to certain kernel functions.
\begin{lemma}\label{lem:lem4.4}
    Let $\omega\left(x\right) $ be a positive function such that $\left(\left\vert \alpha ^{(j)}\right\vert \ast \omega\right)(x) \leq C\omega(x) $ for all $x\in \mathbb{R}$, and $j=1, 3$. Suppose that $\varphi \omega^{-1}\in L^{\infty }\left(\mathbb{R}\right) $. The solution $u\in C^{1}\left([0,T], H^{s}(\mathbb{R})\right)$ of (\ref{eq:cont1})-(\ref{eq:initial}) then satisfies the estimate
    \begin{equation}
        \left\vert u(x,t)\right\vert \leq C\omega\left( x\right)
    \end{equation}
    for all  $x\in \mathbb{R}$,  $t\in \left[ 0,T\right] $.
\end{lemma}
The proof of this lemma  follows a similar idea used in \cite{Bona1981}. We skip the proof and refer the reader to \cite{Erbay2018} (Lemma B.1 in Appendix B)  for a similar proof.

Next, as an application of the above lemma, we consider the  kernel functions (\ref{eq:ros-ker}) and  (\ref{eq:ros-bbm-ker}) corresponding to the Rosenau-KdV equation  (\ref{eq:rosenau}) and the Rosenau-BBM-KdV equation (\ref{eq:rosenau-bbm}), respectively,  both of them are members of the class (\ref{eq:cont})
\begin{example}\label{ex:ex4.5}
        From  (\ref{eq:ros-ker}) and (\ref{eq:ros-bbm-ker}) we get the following inequality
        \begin{displaymath}
                \big\vert \alpha^{(j)}(x)\big\vert \leq C_{1}e^{-a\vert x \vert},~~~~j=1,3
        \end{displaymath}
        with $a=1/\sqrt{2}$ for (\ref{eq:ros-ker}) and $a=\sqrt{3}/2$ for (\ref{eq:ros-bbm-ker}).   Now, taking $\omega(x) =e^{-ra \vert x\vert}$ with any $0<r<1$, we find
        \begin{displaymath}
        \left(\big\vert \alpha^{(j)}\big\vert \ast \omega\right)(x)
                        \leq C_{1}\int_{\mathbb{R}} e^{-a \vert x-y\vert}  e^{-ra \vert y\vert} dy
                       \leq C_{2}e^{-ra\vert x\vert}
         \end{displaymath}
         for $j=1,3$.     Thus, by   Lemma \ref{lem:lem4.4} we deduce that, for initial data satisfying $\varphi ( x) e^{ra\vert x \vert}\in
         L^{\infty }\left( \mathbb{R}\right) $, solutions of  (\ref{eq:rosenau}) and   (\ref{eq:rosenau-bbm})  will satisfy
        \begin{equation}
                \left \vert u(x,t)\right \vert \leq Ce^{-ra\vert x\vert}, ~~~~~0<r<1  \label{Ros-solution}
        \end{equation}
        for all $t\in \left[ 0,T\right]$.
\end{example}

\setcounter{equation}{0}
\section{Numerical Experiments}\label{sec:sec5}

In this section we  confirm the theoretical results with numerical experiments. Through the numerical experiments reported below, we consider particular kernel functions satisfying Conditions C1 and C2 described earlier.

Before  proceeding to the sample cases we briefly address several issues regarding the content and purpose of the numerical experiments to be reported in this section. The main results accomplished in the previous sections are: (i) the proposed numerical method converges with  the quadratic rate of convergence and (ii) the cut-off  error  resulting from taking a finite computational  interval is kept at a certain level.  The principal motivation of the experiments reported below is to check the optimality of the error estimates and to illustrate the variation of the cut-off error. We remind that the explicit solutions to the initial-value problem  of the nonlocal equation are available in the literature only in very special cases of the kernel function. This is the only reason why we consider Rosenau-type equations for the experiments in the following  two subsections. That is,  Rosenau-type equations are considered here just to illustrate how the present method works for a familiar example and from a computational point of view there may be more efficient  numerical methods suggested Rosenau-type equations.  In the last subsection we consider a genuinely nonlocal equation with the Gaussian kernel, for which explicit  solutions are not available. All the numerical experiments confirm the optimality of our error estimates.

As in \cite{Erbay2018, Erbay2021},  there will be no stability limitation regarding spatial mesh size. Roughly speaking, a fully discrete version of our numerical scheme will be straightforward. This is the main reason why we prefer to use an ODE solver rather than a fully discrete scheme in all the numerical experiments. The integration of (\ref{eq:trunca}) in time was done using the  Matlab ODE solver \verb"ode45" based on the fourth-order Runge-Kutta method. To  keep temporal errors  much smaller than spatial errors, the relative and absolute tolerances for the solver \verb"ode45" are chosen to be $RelTol=10^{-10}$ and $AbsTol=10^{-10}$, respectively.

\subsection{The Rosenau-KdV Equation}

We begin our numerical experiments by considering the kernel  (\ref{eq:ros-ker}) since    an exact solution of       (\ref{eq:rosenau}) is available and a decay estimate       for (\ref{eq:ros-ker}) has been already proved in Example \ref{ex:ex4.5}. The Rosenau-KdV equation (\ref{eq:rosenau}) with $g(u)=u^{2}/2$ and $\kappa=1$ admits a solitary wave solution of the form
\begin{equation}
        u(x,t)=A~\text{sech}^{4}\big(B(x-ct) \big), \label{eq:solitary}
\end{equation}
with
\begin{equation}
        A=-\frac{35}{24}+\frac{35}{312}\sqrt{313},~~~~B=\frac{1}{24}\sqrt{-26+2\sqrt{313}},~~~~c=\frac{1}{2}+\frac{1}{26}\sqrt{313} \label{eq:parKdV}
\end{equation}
\cite{Wang2019, Zuo2009}.  The solitary wave  (\ref{eq:solitary}) is initially located at $0$ and propagates to the right  with the constant wave speed $c$. To compare the numerical and exact solutions,  we   solve (\ref{eq:trunca}) with the initial data   $u(x,0)=A\, \text{sech}^{4}\big(B x) \big)$
using the Matlab ODE solver \verb"ode45". Here  the computational domain  is chosen to be  $[-40,80]$ while  the grid spacing is chosen as $h=0.5$ for which $N=120$.  The exact and numerical solutions at  $ t= 40$   are shown in Figure \ref{fig:Fig1} and  there is no noticeable difference between the exact and  numerical results. This numerical experiment clearly indicates that our semi-discrete  scheme  is able to capture  the evolution of the solitary wave   on a relatively coarse mesh for relatively long times.
\begin{figure}[h!]
    \centering
    \includegraphics[width=0.80\linewidth,scale=1.50,keepaspectratio]{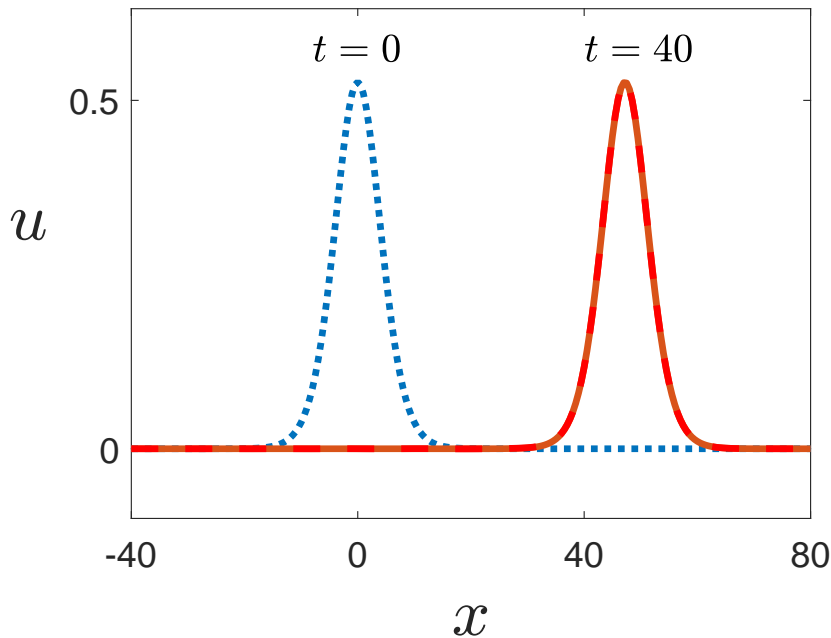}
    \caption{Propagation of a right-moving solitary wave  for  (\ref{eq:cont}) when  the kernel $\alpha$ is given by (\ref{eq:ros-ker}) (that is, the case of the  Rosenau-KdV equation) and $f(u)=u+u^{2}/2$, $\kappa=1$. The initial profile, the exact and the numerical solutions at   $t=40$   are shown with the dotted line, the solid line and the dashed line, respectively. The numerical solution is almost indistinguishable from the exact solution.
    The computational interval, the number of grid points and  the mesh size are $[-40,80]$,  $N=120$  and   $h=0.5$, respectively.}
    \label{fig:Fig1}
\end{figure}

To verify the convergence rate estimate derived in Theorem \ref{theo:theo3.3}  for the spatial discretization error, we now perform numerical experiments  for the initial-value problem defined above.  The computational domain is taken so large  to avoid influence of the localization error brought on by the finite domain size. The  $l^{\infty}$-errors $E_{h}^{N}$ at time $t$  are computed using the formula
\begin{equation}
        E_{h}^{N}(t)=\left\Vert \mathbf{u}(t) - \mathbf{u}_{h}^{N}(t) \right\Vert_{l^{\infty}}
                =\max_{-N\leq i \leq N} \left \vert u(x_{i},t)- (\mathbf{u}_{h}^{N})_{i}\right \vert . \label{eq:linferror}
 \end{equation}
The experimental convergence rate $\rho$ is calculated by the formula
 \begin{equation}
        \rho={{\log\left({{E_{h_{1}}^{N}(t)}/ {E_{h_{2}}^{N}(t)}}\right)}\over {\log\left({h_{1}}/ {h_{2}}\right)}},
 \end{equation}
using the errors at two different values $h_{1}$ and $h_{2}$ of the mesh size.

 In Figure \ref{fig:Fig2}  we present   the errors  measured using (\ref{eq:linferror}) for different mesh sizes, where the mesh size varies from $h =1$ to $h =2^{-5}$ and the computational interval is $[-100,100]$. The errors  calculated at time $t =40$  are plotted versus $h$ on a logarithmic scale (the solid line with circle markers).    For comparison purposes,  the theoretical quadratic convergence in space is indicated by a dashed line. We observe that, when   the kernel is given by (\ref{eq:ros-ker}), the experimental rate of convergence corroborate the quadratic order of convergence proved in Theorem \ref{theo:theo3.3} for the semi-discrete scheme .

\begin{figure}[h!]
    \centering
    \includegraphics[width=0.80\linewidth,scale=1.50,keepaspectratio]{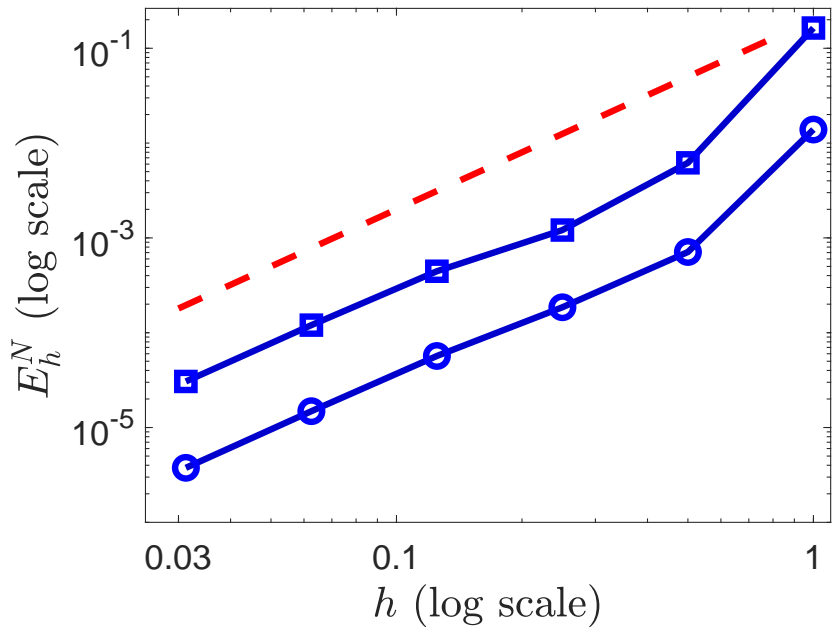}
    \caption{Variation of the error at   $t=40$ with the mesh size $h$.   The theoretical quadratic convergence is plotted as a dashed line with slope 2:1 for reference. The circle and square markers indicate the data points of the numerical experiments for the  problem of finding a solitary wave solution  of  the nonlocal nonlinear wave equation (\ref{eq:cont}).   The circle markers are for the case where the kernel $\alpha$ is given by (\ref{eq:ros-ker}) (that is, the case of the  Rosenau-KdV equation).  The square markers are for the case where the kernel $\alpha$ is given by (\ref{eq:ros-bbm-ker}) (that is, the case of the  Rosenau-BBM-KdV equation).     For both cases,  $f(u)=u+u^{2}/2$ and $\kappa=1$.  }
    \label{fig:Fig2}
\end{figure}

In the above experiments, the numerical results are obtained by solving the $2N+1$ equations of (\ref{eq:trunca}), that were obtained by truncating  the infinite equations system  (\ref{eq:disc}).  To investigate how this truncation  affects the numerical results we conduct another set of the numerical experiments. The mesh size is  fixed  at $h =0.05$ but the number of grid points is steadily increased.  Consequently ,   the size of the computational domain will not be the same in all experiments.  Figure \ref{fig:Fig3}  shows,  on a semi-logarithmic scale, the variation of the error at time $t=40$ with $N$ for  sufficiently large numbers of grid points  (the  line with circle markers).  One can clearly see that, as $N$ increases, the error first decreases and then stagnates.  In other words, up to a certain value of $N$ ($ \approx 10^3$), the localization error dominates   the total error in the theoretical error estimate $E_{h}^{N}={\mathcal O}\left(h^{2}+\epsilon\right)$ and then  the spatial discretization error dominates. This is the expected outcome and is in line with the decay estimate in Example \ref{ex:ex4.5}. Note that that the solitary wave  solution in (\ref{eq:solitary}) decays exponentially to zero for $|x| \rightarrow \infty$ and  $\epsilon={\mathcal O}\left(e^{-CNh}\right)$  by Example \ref{ex:ex4.5}. The numerical experiments in  Figure \ref{fig:Fig3}  do confirm that the level of the localization error can be controlled if a large enough $N$ is used.

\begin{figure}[h!]
    \centering
    \includegraphics[width=0.80\linewidth,scale=1.50,keepaspectratio]{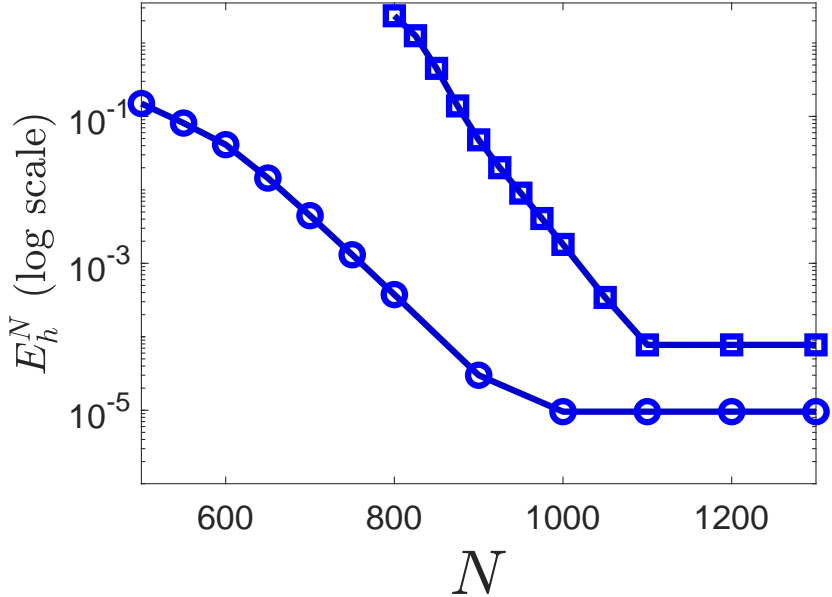}
    \caption{Variation of  the $l^{\infty}$-error ($E_{h}^{N}$) at   $t=40$  with $N$ for the solitary wave problem  of the nonlocal nonlinear wave equation (\ref{eq:cont}). The circle and square markers indicate the data points of the numerical experiments.  The circle markers are for the case where the kernel $\alpha$ is given by (\ref{eq:ros-ker}) (that is, the case of the  Rosenau-KdV equation).  The square markers are for the case where the kernel $\alpha$ is given by (\ref{eq:ros-bbm-ker}) (that is, the case of the  Rosenau-BBM-KdV equation).     For both cases,  $f(u)=u+u^{2}/2$ and $\kappa=1$.      The  mesh size  is fixed at $h=0.05$. The computational domain is $[-Nh,  Nh]$. }
    \label{fig:Fig3}
\end{figure}

\subsection{The Rosenau-BBM-KdV Equation}

This subsection will provide further numerical experiments to  address the performance of the  semi-discrete scheme for the kernel $\alpha$ given by (\ref{eq:ros-bbm-ker}).   When $g(u)=u^{2}/2$ and $\kappa=1$, an exact solution of (\ref{eq:rosenau-bbm})  was given in \cite{ Wongsaijai2014} in the form   (\ref{eq:solitary})   with
\begin{equation}
        A=\frac{5}{456}(-25+13\sqrt{457}),~~~~B=\frac{1}{\sqrt{288}}\sqrt{-13+\sqrt{457}},~~~~c=\frac{241+13\sqrt{457}}{266}. \label{eq:parBBM}
\end{equation}
Using the same approach as in the previous subsection, we now estimate  the convergence rate of the semi-discrete scheme applied for (\ref{eq:ros-bbm-ker})  and discuss the decay of the localization errors.

Consider the initial-value problem defined by  (\ref{eq:trunca}) with the kernel (\ref{eq:ros-bbm-ker}) and the initial data   $u(x,0)=A\, \text{sech}^{4}\big(B x) \big)$ where $A$ and $B$ are given by (\ref{eq:parBBM}). Again, we use the Matlab ODE solver \verb"ode45" to solve the truncated problem.    In the first set of the numerical experiments,  the computational interval $[-80, 120]$ is large enough to ensure that the  localization errors are negligible compared to the discretization errors. The mesh size $h$ changes from  $h =1$ to $h =2^{-5}$.   In Figure \ref{fig:Fig2}  we plot,   on a logarithmic scale, the spatial discretization errors  (the solid line with square markers)    measured at time $t =40$   using (\ref{eq:linferror}). Figure \ref{fig:Fig2} clearly shows that the convergence rate obtained  from the numerical experiments is  in excellent agreement with the theoretical quadratic convergence rate of Theorem \ref{theo:theo3.3}, indicated by a dashed line. As can be seen from  Figure \ref{fig:Fig2}, for a fixed mesh size,  the errors corresponding to the Rosenau-BBM-KdV equation  are slightly higher than those corresponding to the Rosenau-KdV equation. This is due to that the value of the amplitude parameter $A$ in the present problem is approximately five times grater than the value of $A$  in the previous subsection. In other words, the present problem is more nonlinear than the previous one.

In the second set of the numerical experiments, we fix  the grid spacing $h=0.05$, and vary the number of grid points $2N+1$.   The errors at time $t=40$ are displayed in  Figure \ref{fig:Fig3},  on a semi-logarithmic scale (the  line with square markers).  In the figure we observe the behavior seen for the solitary wave problem of the Rosenau-KdV equation and displayed again in Figure \ref{fig:Fig3}.  It shows that,  for relatively small values of $N$, the localization error  contributes more to the total error  but it disappears as $N$ increases above some critical value ($\approx 1100$).  This is due to the exponentially decaying nature of the solitary wave solution.   So, in the case of (\ref{eq:ros-bbm-ker}) too, we get the conclusion that  the localization error  can be  made negligible relative to the discretization error  if $N$ is large enough.

\begin{figure}[h!]
    \centering
    \includegraphics[width=0.80\linewidth,scale=1.50,keepaspectratio]{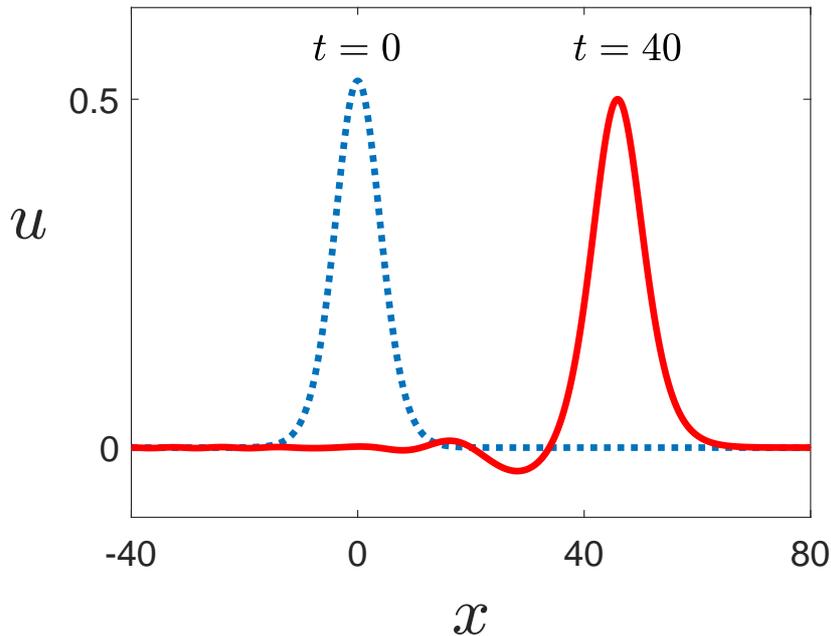}
    \caption{Propagation of a right-moving solitary wave  for  (\ref{eq:cont}) with $\alpha(x)= \frac{1}{\sqrt{2\pi}}e^{-\frac{x^2}{2}}$, $f(u)=u+u^{2}/2$ and $\kappa=1$. The initial profile and the numerical solutions at   $t=40$   are shown with the dotted line and the solid line, respectively.     The computational interval, the number of grid points and  the mesh size are $[-40,80]$,  $N=1200$  and   $h=0.05$, respectively.  }
    \label{fig:Fig4}
\end{figure}

\subsection{Gaussian Kernel}

In this subsection we will report about convergence rate  in the case when the kernel function is not  the Green's function of a differential operator. We stress that the kernels considered in the previous two subsections  are the opposite of the present situation. Let the kernel $\alpha$ be  the Gaussian kernel
\begin{equation}
    \alpha(x)= {1\over {\sqrt{2\pi}}}e^{-{x^2\over  2}}, \label{eq:gauss}
\end{equation}
which is a commonly used  and  infinitely smooth kernel. The following experiments are performed using (\ref{eq:gauss}).

As in the previous two subsections, we solve (\ref{eq:trunca}) with   the initial data   $u(x,0)=A\, \text{sech}^{4}(B x)$ for the time interval  $[0, 40]$. We assume that the the constants $A$ and $B$ are given by (\ref{eq:parKdV}).  We first note that the decay  estimate  (\ref{Ros-solution}) derived for solutions of the Rosenau-KdV equation  and the Rosenau-BBM-KdV equation using Lemma \ref{lem:lem4.4} is also valid for the solution of the present problem.   The reasons for this situation are twofold. On the one hand, we accept the same initial condition as before.  On the other hand,   the function $e^{-a\vert x\vert }$ in Example  \ref{ex:ex4.5}   dominates the Gaussian kernel and its derivatives.

The initial profile and the numerical solution at $t=40$ obtained for  the computational interval $[-40,80]$ and the mesh size $h=0.05$    are  illustrated in Figure \ref{fig:Fig4}.   The figure shows oscillatory wavetrain due to unbalanced  dispersive regularization. That is, the solution profile is  different from those in the previous two subsections. The reason is that the kernel functions and consequently  the dispersive effects are taken differently. So, in the present experiment, we cannot expect to observe a solitary wave  which is generated by the balance between the nonlinear and dispersive effects.

\begin{table}[tbhp]
    \caption{ Refinement levels and corresponding experimental  orders of convergence  for the single solitary wave problem of   (\ref{eq:cont}) with      $\alpha(x)= \frac{1}{\sqrt{2\pi}}e^{-\frac{x^2}{2}}$, $f(u)=u+u^{2}/2$ and $\kappa=1$.}
    \label{tab:tab1}
    \vspace*{10pt}
   \centering
   \begin{tabular}{|c|c|c|} \hline
   $h$                    &  $N$                  &  Order of convergence  ($\rho$)         \\ \hline
   $1$                         &  $120 $                &  -   \\
   $2^{-1}$               &  $240  $                &  -    \\
  $ 2^{-2}$               &  $480  $              &     $2.01427444$     \\
  $ 2^{-3}$               &  $960 $               &     $2.00294376$     \\
   $2^{-4} $              &  $1920 $              &    $ 2.00084024$        \\
   $2^{-5}$               &  $3840  $             &  $2.00019716$        \\
   \hline
   \end{tabular}
\end{table}

To investigate  how fast the discretization error vanishes as $h$  decrease, we perform numerical experiments for various values of $h$. Since no exact solution is known for the initial-value problem, we compute the experimental order of convergence by the formula
 \begin{equation}
        \rho=\frac{1}{\log 2}\log\Big(\frac{\Vert \mathbf{u}_{h}^{N}(t) - \mathbf{u}_{h/2}^{N}(t)\Vert_{l^{\infty}}}{\Vert \mathbf{u}_{h/2}^{N}(t)- \mathbf{u}_{h/4}^{N}(t)\Vert_{l^{\infty}}}\Big)
 \end{equation}
using the approximate solutions obtained  for three successive values $h$, $h/2$ and $h/4$ of the mesh size.  The mesh size  varies from  $h =1$ to $h =2^{-5}$ for the computational interval $[-110, 130]$ and the numerical solutions corresponding to  $t=40$ are used to calculate the experimental order of convergence.  From the results in Table \ref{tab:tab1} we see that the experimental rate of convergence is perfectly consistent with the   quadratic convergence estimate established in Theorem  \ref{theo:theo3.3}.    That is, the convergence rate observed here is the same as the convergence rates observed in the previous two subsections even though the kernel functions considered  have completely different characters.

\bibliographystyle{plainnat}
\bibliography{Erbay-arXiv-02-02-2022}

\end{document}